\documentclass[a4paper,12pt,reqno]{amsart}
\usepackage{fullpage}
\usepackage{mymacros}
\usepackage{oitp}
\usepackage[all]{xy}
\usepackage{stmaryrd,mdwlist}
\title{An example of birationally inequivalent projective symplectic varieties which are D-equivalent and L-equivalent}

\author{Shinnosuke Okawa}
\address{
Department of Mathematics,
Graduate School of Science,
Osaka University,
Machikaneyama 1-1,
Toyonaka,
Osaka,
560-0043,
Japan.
}
\email{okawa@math.sci.osaka-u.ac.jp}

\begin{document}

\maketitle

\begin{abstract}
We give an example of a pair of projective symplectic varieties in arbitrarily large dimensions which are D-equivalent, L-equivalent, and birationally inequivalent.
\end{abstract}

%
%
%
%\tableofcontents

%
%
%
\section{Introduction}\label{sc:Introduction}

It is widely believed that the bounded derived category of coherent sheaves
$
 \bfD \lb X \rb = \bfD ^{ b } \coh X
$ is a fundamental invariant of a smooth projective variety $X$. It is hence natural to ask which kind of information of the variety $X$ can be regained from the triangulated category $\bfD \lb X \rb$.

The \emph{Grothendieck ring of varieties},
which will be denoted by
$
 K _{ 0 } \lb \Var \rb
$
in this paper, is the quotient of the free abelian group generated by the set of
isomorphism classes of schemes of finite type over the fixed base field
$
 \bfk
$
modulo the relations
\begin{equation}
 [ X ] = [ X \setminus Z ] + [ Z ]
\end{equation}
for closed embeddings
$
 Z \subset X.
$
Multiplication in
$
 K _{ 0 } \lb \Var \rb
$
is defined by the Cartesian product,
which is easily seen to be associative, commutative, and unital with
$
 1 = [ \Spec \bfk ].
$
The \emph{localized Grothendieck ring of varieties}
$
 K _{ 0 } \lb \Var \rb [ \bL ^{ - 1 } ]
$
is the localization of
$
 K _{ 0 } \lb \Var \rb
$
by the class
$
 \bL = [ \bA ^{ 1 } ]
$
of the affine line.

A pair
$
 \lb X, Y \rb
$
of smooth projective varieties are said to be D-equivalent if they have equivalent derived categories. Similarly, they are said to be L-equivalent if they satisfy the following equivalent conditions.
\begin{align}
 [ X ] = [ Y ] \in K _{ 0 } \lb \Var \rb [ \bL ^{ - 1 } ]
 \iff
 \bL ^{ m } \cdot \lb [ X ] - [ Y ] \rb = 0 \in K _{ 0 } \lb \Var \rb
 \quad \exists m \in \bN
\end{align}

It is asked independently in the first preprint versions of \cite{Kuznetsov2017} and \cite{Ito:2016ac} if D-equivalence should imply L-equivalence. This is motivated by the first such example found in \cite{Borisov:2014aa} \cite{MR3535349} and the other examples discovered in \cite{Ito:2016ab} (the D-equivalence is shown later in\cite{Kuznetsov:2016ab}), \cite{Kuznetsov2017}, and \cite{Hassett:2016aa} \cite{Ito:2016ac}. In addition, more supporting evidences have been discovered in the works \cite{Borisov:2017aa}, \cite{Kapustka:2017aa}, \cite{Manivel:2017aa}, and \cite{Kapustka:2017ab}. In fact, all known examples are pairs of simply connected Calabi-Yau varieties with
$
 h ^{ 2, 0 } = 0
$
or K3 surfaces. On the other hand, it is shown in \cite{Ito:2016ac} and \cite{Efimov:2017aa} that the pair of an abelian variety of dimension at least two and its dual is a counter-example to the question as soon as the endomorphism ring of the abelian variety is isomorphic to
$
 \bZ
$. Taking these counter-examples into account, the simply-connectedness is assumed in \cite[Conjecture 1.6]{Kuznetsov2017}. \cite[Section 7]{Ito:2016ac} instead proposes to modify the Grothendieck ring of varieties suitably.

The aim of this paper is to give a first example of a pair of projective symplectic varieties which are both D-equivalent, L-equivalent, and birationally inequivalent in arbitrarily large dimension. Let
$
 X
$
be a projective K3 surface. It is well known by \cite{10.2307/2373734} that the Hilbert scheme
$
 X ^{ [ n ] }
$
of $n$ points on $X$, which very roughly is described as
\begin{align}
 X ^{ [ n ] } =
 \lc I \subset \cO _{ X } \mid \dim _{ \bfk } \cO _{ X } / I = n \rc,
\end{align}
is a smooth projective symplectic variety of dimension $2 n$.
It is simply connected and satisfies
$
 h ^{ 2, 0 } \lb X \rb = 1
$. Below is the main result of this article.

\begin{theorem}\label{th:main}
Let
$
 \lb X, Y \rb
$
be a non-isomorphic pair of K3 surfaces of Picard number 1 and of degree
$
 2 d _{ X }, 2 d _{ Y }
$
respectively which are both D-equivalent and L-equivalent. Then
$
 X ^{ [ n ] }
$
and
$
 Y ^{ [ n ] }
$
are D-equivalent, L-equivalent, and birationally inequivalent if either
\begin{enumerate}[(1)]
\item
$
 d _{ X } \neq d _{ Y }
$
or

\item\label{it:equality}
$
 d _{ X } = d _{ Y }
$,
$
 n >2
$,
and there exists an integer solution to the following Pell's equation.
\begin{align}\label{eq:Pell's(b)}
 \lb n - 1 \rb X ^{ 2 } - d _{ X } Y ^{ 2 } = 1
\end{align}
\end{enumerate}
\end{theorem}

\begin{remark}
It follows either from \cite[Theorem 4.1]{Hassett:2016aa} or \cite[Theorem 1.3]{Ito:2016ac} that if $X$ is a very general K3 surface of degree 12 (i.e., $d _{ X } = 6$) and
$Y$ is \emph{the} Fourier-Mukai partner of $X$ ($d _{ Y } = 6$), then
$
 X
$
and
$
 Y
$
are L-equivalent. Hence, e.g., if
$
 n = 6 y ^{ 2 } + 2
$
for some positive integer $y$, the assumption (\pref{it:equality}) is satisfied by the obvious solution
$
 \lb X, Y \rb = \lb 1, y \rb
$.
Hence we do have an example as in \pref{th:main} for
$
 n = 6 y ^{ 2 } + 2 \lb y = 1, 2, \dots \rb
 = 8, 26, 56, \dots
$.

In fact, as we mention in the next remark, one can determine if
$
 X ^{ [ n ] }
$
and
$
 Y ^{ [ n ] }
$
are birationally equivalent or not by checking the existence of solutions to certain set of Pell's equation. Hence for those $n$ where the birationality holds,
we do not have an example in dimension $2 n$ yet.
Note, however, if there is a Fourier-Mukai pair
$
 ( X, Y )
$
of K3 surfaces of Picard number one which are L-equivalent and
$
 d _{ X } \neq d _{ Y }
$,
then by \pref{th:main} we can construct examples for every $n$.
\end{remark}

\begin{remark}\label{rm:finer}
In the recent preprint \cite{Meachan:2018aa}, the authors gave a criterion for the birationality of the Hilbert schemes of points
$
 X ^{ [ n ] }
$
and
$
 Y ^{ [ n ] }
$
on K3 surfaces
$
 X, Y
$
of Picard number one. In \cite[Proposition 1.2]{Meachan:2018aa} they apply the criterion to the pair
$
 \lb X, Y \rb
$
which appeared in the previous remark, to construct examples of pairs of Hilbert schemes of points
which are D-equivalent and birationally inequivalent.
They in particular show that
$
 X ^{ [ n ] }
$
and
$
 Y ^{ [ n ] }
$
are birationally equivalent to each other for some $n$, starting with
$
 n = 2, 3, 4, 10, 12, 14, 15, 18, 20, \dots
$.
In fact, they are even isomorphic to each other for
$
 n = 2
$
and
$
 3
$
as shown in \cite[Example 7.2]{MR1872531}.

%See also \pref{rm:there_are_birational_examples}.

\end{remark}

The proof of \pref{th:main} is given in the next section. In fact the D-equivalence is nothing but \cite[Proposition 8]{MR2353249}, which in turn is an application of \cite{Bridgeland-King-Reid}.
The L-equivalence immediately follows from the description of the generating series of the Hilbert scheme of points on a smooth quasi-projective variety $Z$ as the $ ( \bL ^{ - \dim Z } [ Z ] )$-th (!) power of the generating series for $ \bA ^{ \dim X }$ due to \cite{gusein-zade2006}. In order to show the birational inequivalence of
$
 X ^{ [ n ] }
$
and
$
 Y ^{ [ n ] }
$,
we use the very detailed description of the movable cone of
$
 X ^{ [ n ] }
$
due to \cite{MR3279532}. This is the only place where we use the assumption on the Picard number and the degrees, and will be discussed in \pref{pr:main_prop}.

Throughout this paper the base field $\bfk$ will be the field of complex numbers
$
 \bC
$. Varieties are always assumed to be connected.

\begin{acknowledgements*}
The author is indebted to Kota Yoshioka, Michal Kapustka, and Grzegorz Kapustka for pointing out a crucial error in Proposition 2.2 of the first draft and informing him of references. He is also greatful to Kota Yoshioka for sending the preprint \cite{Meachan:2018aa} to the author.
The author was partially supported by Grants-in-Aid for Scientific Research
(16H05994,
16K13746,
16H02141,
16K13743,
16K13755,
16H06337)
and the Inamori Foundation.
\end{acknowledgements*}

\section{Proof of \pref{th:main}}

Let
$
 \lb X, Y \rb
$
be a pair of smooth projective surfaces which are both D-equivalent and
L-equivalent. Then
$
 \lb X ^{ [ n ] }, Y ^{ [ n ] } \rb
$
is a pair of smooth projective $2 n$-folds by \cite{MR0237496}, which are also D-equivalent by \cite[Proposition 8]{MR2353249}.
On the other hand, one can show the L-equivalence as follows. This is essentially due to
\cite{gusein-zade2006}, and it applies to smooth quasi-projective varieties of arbitrary dimension.
In the proof we consider the generating series of the Hilbert scheme, which is defined for an arbitrary variety $Z$ as follows.
\begin{align}
 \bH _{ Z } ( T ) \coloneqq \sum _{ n = 0 } ^{ \infty } \ld Z ^{ [ n ] } \rd T ^{ n }
 = 1 + [ Z ] T + \cdots
 \in
 1 + T \cdot K _{ 0 } \lb \Var \rb \ldd T \rdd.
\end{align}

\begin{lemma}\label{lm:main}
Let
$
 \lb X, Y \rb
$
be a pair of L-equivalent smooth quasi-projective varieties. Then
$
 \lb X ^{ [ n ] }, Y ^{ [ n ] } \rb
$
also is a pair of L-equivalent varieties.
\end{lemma}

\begin{proof}
%It is well known that the Hilbert scheme of $n$ points of a smooth surface is again smooth of dimension $2 n$, and that the Hilbert-Chow morphism
%$
% X ^{ [ n ] } \to \Sym ^{ n } X
%$
%is a crepant resolution of
%$
% \Sym ^{ n } X
%$.
%Moreover there is a derived equivalence
%\begin{align}\label{eq:Hilb_vs_quotient_stack}
%\begin{autobreak}
% \bfD [ X ^{ n } / \frakS _{ n } ] \simeq \bfD X ^{ [ n ] },
%\end{autobreak}
%\end{align}
%which is linear over
%$
% \Sym ^{ n } X
%$.
%
%By the assumption, there exists a kernel
%$
% K \in \bfD X \times Y
%$
%which yields the Fourier-Mukai transformation
%\begin{align}
% \Phi _{ K } ^{ X \to Y } \colon \bfD X \simto \bfD Y.
%\end{align}
%We can use this kernel to produce the object
%\begin{align}
% K ^{ \boxtimes n } \in \bfD ^{ \frakS _{ n } } \lb X ^{ n } \times Y ^{ n } \rb,
%\end{align}
%by which we obtain the derived equivalence
%\begin{align}
% \Phi _{ K ^{ \boxtimes n } } \colon
% \bfD \ld X ^{ n } / \frakS _{ n } \rd \simto \bfD \ld Y ^{ n } / \frakS _{ n } \rd
%\end{align}
%By composing it with the equivalence \eqref{eq:Hilb_vs_quotient_stack} (for both
%$
% X
%$
%and
%$
% Y
%$), we obtain the K-equivalence of
%$
% X ^{ [ n ] }
%$
%and
%$
% Y ^{ [ n ] }
%$.
Let
$
 m
$
be a natural number such that
\begin{align}
 \bL ^{ m } \cdot \lb [ X ] - [ Y ] \rb = 0
 \iff
 \bL ^{ m } \cdot [ X ] = \bL ^{ m } \cdot [ Y ] \in K _{ 0 } \lb \Var \rb.
\end{align}
Since both
$
 \bL ^{ m } \cdot [ X ]
$
and
$
 \bL ^{ m } \cdot [ Y ]
$
are primitive elements in the sense of \cite[Definition 1.1]{MR2004431},
their dimensions are well-defined and the same. Hence we see
$
 \dim X = \dim Y
$.

By \cite[COROLLARY]{gusein-zade2006}, for any smooth quasi-projective variety
$
 Z
$
there is an equality
\begin{align}
 \bH _{ Z } ( T ) =
 \lb \bH _{ \bA ^{ \dim Z } } ( T ) \rb ^{ \bL ^{ - \dim Z } \ld Z \rd }
 \in K _{ 0 } \lb \Var \rb \ld \bL ^{ - 1 } \rd \ldd T \rdd
\end{align}
(see \cite{gusein-zade2006} and references therein for the notion of the power structure of
$
 K _{ 0 } \lb \Var \rb
$).
Combining it with the assumption
\begin{align}\label{eq:L-equivalence}
 \ld X \rd = \ld Y \rd \in K _{ 0 } \lb \Var \rb \ld \bL ^{ - 1 } \rd,
\end{align}
we obtain the following sequence of equalities.
\begin{align}
\begin{autobreak}
 \bH _{ X } ( T )
 =
 \lb \bH _{ \bA ^{ \dim X } } ( T ) \rb ^{ \bL ^{ - \dim X } \ld X \rd }
 =
 \lb \bH _{ \bA ^{ \dim Y } } ( T ) \rb ^{ \bL ^{ - \dim Y } \ld Y \rd }
 =
 \bH _{ Y } ( T )
 \in K _{ 0 } \lb \Var \rb \ld \bL ^{ - 1 } \rd \ldd T \rdd. 
\end{autobreak}
\end{align}
Comparing the coefficients of
$
 T ^{ n }
$,
we obtain the L-equivalence of
$
 X ^{ [ n ] }
$
and
$
 Y ^{ [ n ] }
$.
\end{proof}

Let us now specialize to the pair
$
 \lb X, Y \rb
$
as in \pref{th:main}. In the rest we assume
$
 n \ge 2
$
(for the case $n=1$ being trivial). Since we assumed that
$
 X
$
is of Picard number 1, we can and will use the results in \cite[Section 13]{MR3279532} to understand the movable cone of
$
 X ^{ [ n ] }
$.

\begin{proposition}\label{pr:main_prop}
Let $X$ and $Y$ be a pair of non-isomorphic K3 surfaces of Picard number 1 and of degree
$
 2 d _{ X }, 2 d _{ Y }
$
respectively. Then
$
 X ^{ [ n ] }
$
and
$
 Y ^{ [ n ] }
$
are not birationally equivalent if either
\begin{enumerate}[(1)]
\item\label{it:inequality}
$
 d _{ X } \neq d _{ Y }
$
or

\item\label{it:equality_revisited}
$
 d _{ X } = d _{ Y }
$,
$
 n >2
$,
and there exists an integer solution to the following Pell's equation.
\begin{align}\label{eq:Pell's(b)_revisited}
 \lb n - 1 \rb X ^{ 2 } - d _{ X } Y ^{ 2 } = 1
\end{align}
\end{enumerate}
\end{proposition}

\begin{proof}
Let us briefly recall the results in \cite[Section 13]{MR3279532}.
The Picard group
$
 \Pic \lb X ^{ [ n ] } \rb
$
is freely generated by the two divisors
$
 \Htilde
$
and
$
 B
$,
where
$
 \Htilde
$
is the pull-back of the ample generator of
$
 \Pic \lb \Sym ^{ n } X \rb
$
by the Hilbert-Chow morphism,
%$
% \pi \colon X ^{ [ n ] } \to \Sym ^{ n } X
%$,
and
$
 B
$
is the half of the exceptional divisor.
% of
%$
% \pi
%$.
Moreover there exists a primitive embedding
\begin{align}
 \Pic \lb X ^{ [ n ] } \rb \hookrightarrow H ^{ * } \lb X, \bZ \rb,
\end{align}
where
$
 H ^{ * } \lb X, \bZ \rb
 =
 H ^{ 0 } \lb X, \bZ \rb \oplus \Pic \lb X \rb \oplus H ^{ 4 } \lb X, \bZ \rb
$
is the Mukai lattice of $X$ equipped with the Mukai pairing
\begin{align}
 \lb r, L, s \rb \cdot \lb r ', L ', s ' \rb
 =
 L L ' - r s ' - s r ' \in \bZ.
\end{align}
The embedding is an isometry with respect to this pairing and
the Beauville-Bogomolov-Fujiki form
$
 q
$
on
$
 \Pic \lb X ^{ [ n ] } \rb
$, and the embedding sends
$
 \Htilde
$
to
$
 \lb 0, - H, 0 \rb
$
and
$
 B
$
to
$
 \lb - 1, 0, 1 - n \rb
$.

As explained in \cite[Proposition 13.1]{MR3279532},
there are three possibilities (a), (b), and (c) for the two extremal rays of the movable cone
$
 \Mov \lb X ^{ [ n ] } \rb
$.
In any case, one of the rays is spanned by the primitive vector
$
 \Htilde
$
corresponding to the Hilbert-Chow morphism.

Suppose for a contradiction that there exists a birational map
$
 \varphi \colon X ^{ [ n ] } \dasharrow Y ^{ [ n ] }
$.
Since both
$
 X ^{ [ n ] }
$
and
$
 Y ^{ [ n ] }
$
are smooth and have trivial canonical bundles,
$
 \varphi
$
is an isomorphism in codimension one. Hence by \cite[Lemma 2.6]{MR1664696} it induces an isometry
\begin{align}\label{eq:the_isometry}
 \varphi _{ * } \colon
 \lb \Pic \lb X ^{ [ n ] } \rb, q _{ X } \rb \simto
 \lb \Pic \lb Y ^{ [ n ] } \rb, q _{ Y } \rb,
\end{align}
which by its construction also respects the movable cones. 
In particular
$
 \varphi _{ * } ^{ - 1 }
$
sends the base point free divisor
$
 \Htilde _{ Y }
$,
the primitive ample divisor on
$
 Y ^{ [ n ] }
$
corresponding to the Hilbert-Chow morphism of $Y$, to either
\begin{enumerate}[(i)]
\item\label{it:nef_cone_preserving}
$
 \Htilde
$
or
\item\label{it:nef_cone_not_preserved}
the primitive generator of the other extremal ray
$
 \rho
$
 of
$
 \Mov \lb X ^{ [ n ] } \rb
$.
\end{enumerate}
In the case (\pref{it:nef_cone_preserving}), the birational map
$
 \varphi
$
respects the exceptional divisors of the Hilbert-Chow morphisms of $X$ and $Y$. Hence by \cite[Theorem 2.1]{MR770463},
$
 \varphi
$
should be induced from an isomorphism from $X$ to $Y$ (note that any birational map between $X$ and $Y$ is an isomorphism). Since
$
 X
$
and
$
 Y
$
are not isomorphic to each other by the assumption, this is a contradiction.

In the rest of the proof we assume (\pref{it:nef_cone_not_preserved}) and show that we end up with a contradiction,
%in any of the three cases (a), (b), and (c),
to conclude that
$
 Y ^{ [ n ] }
$
is birationally inequivalent to
$
 X ^{ [ n ] }
$.
Let us now assume
\begin{align}\label{eq:the_inequality}
 d _{ X } \ge d _{ Y }
\end{align}
without loss of generality.

In the case (a),
$
 \rho
$
corresponds to the (rational) Lagrangian fibration of
$
 X ^{ [ n ] }
$. Hence this case can not occur under our assumptions.

In the case (b),
$
 \rho
$
is spanned by the integral divisor
\begin{align}\label{eq:ext_ray_(b)}
 x _{ 1 } \lb n - 1 \rb \Htilde - d _{ X } y _{ 1 } B,
\end{align}
where
$
 x _{ 1 }, y _{ 1 } > 0
$
is the integer solution of the Pell's equation \eqref{eq:Pell's(b)_revisited} with the smallest
$
 x _{ 1 }
$.
It is easy to see that
$
 \gcd \lb x _{ 1 } \lb n - 1 \rb, d _{ X } y _{ 1 } \rb = 1
$, since otherwise
$
 \lb x _{ 1 }, y _{ 1 } \rb
$
can not be a solution of \eqref{eq:Pell's(b)}.
Hence \eqref{eq:ext_ray_(b)} is the primitive generator of
$
 \rho
$.
Now since \eqref{eq:the_isometry} is an isometry, we obtain the following equality.
\begin{align}
\begin{autobreak}
 2 d _{ Y }
 =
 q _{ Y } \lb \Htilde _{ Y }, \Htilde _{ Y } \rb
 =
 q _{ X } \lb x _{ 1 } \lb n - 1 \rb \Htilde - d _{ X } y _{ 1 } B, x _{ 1 } \lb n - 1 \rb \Htilde - d _{ X } y _{ 1 } B \rb
 =
 \lb x _{ 1 } \rb ^{ 2 } \lb n - 1 \rb ^{ 2 } \lb 2 d _{ X } \rb - \lb d _{ X } \rb ^{ 2 } \lb y _{ 1 } \rb ^{ 2 } \lb 2 \lb n - 1 \rb \rb
 = 2 d _{ X } \lb n - 1 \rb \lb \lb n - 1 \rb \lb x _{ 1 } \rb ^{ 2 } - d _{ X } \lb y _{ 1 } \rb ^{ 2 } \rb
 = 2 d _{ X } \lb n - 1 \rb.
\end{autobreak}
\end{align}
For the last equality, we use that
$
 \lb x _{ 1 }, y _{ 1 } \rb
$
is a solution of the Pell's equation \eqref{eq:Pell's(b)_revisited}. Since it follows from \eqref{eq:the_inequality} that
$
 2 d _{ X } \lb n - 1 \rb \ge 2 d _{ Y }
$,
we should have
$
 n = 2
$
and
$
 d _{ X } = d _{ Y }
$.
This contradicts both of the assumptions (\ref{it:inequality}) and (\ref{it:equality_revisited}). Since the existence of a solution to the Pell's equation is assumed in the case (\ref{it:equality_revisited}), here we conclude the proof in that case because of the trichotomy in \cite[Proposition 13.1]{MR3279532}.

Finally, suppose that we are in the case (c) under the assumption (\ref{it:inequality}). Then
$
 \rho
$
is spanned by the integral divisor
\begin{align}\label{eq:primitive_generator_(c)}
 x ' _{ 1 } \Htilde - y ' _{ 1 } d _{ X } B,
\end{align}
where
$
 x ' _{ 1 }, y ' _{ 1 }
$
is the integer solution of the Pell's equation
\begin{align}
 X ^{ 2 } - d _{ X }\lb n - 1 \rb Y ^{ 2 } = 1
\end{align}
with the smallest
$
 \frac{ y ' _{ 1 } }{ x ' _{ 1 } } > 0
$.
One can easily check as in the case (b) that
\eqref{eq:primitive_generator_(c)} is the primitive generator of
$
 \rho
$. Thus we obtain the following contradiction.
\begin{align}
\begin{autobreak}
 2 d _{ Y }
 =
 q _{ Y } \lb \Htilde _{ Y }, \Htilde _{ Y } \rb
 =
 q _{ X }
 \lb x ' _{ 1 } \Htilde - y ' _{ 1 } d _{ X } B, x ' _{ 1 } \Htilde - y ' _{ 1 } d _{ X } B \rb
 =
 2 d _{ X }.
\end{autobreak}
\end{align}
\end{proof}

\begin{remark}\label{rm:there_are_birational_examples}
If
$
 d _{ X } = d _{ Y }
$
and there is no solution to the Pell's equation \eqref{eq:Pell's(b)_revisited}, then we are in the case (c) but get no contradiction. In fact this does occur, e.g., when $n = 3$ and
$
 \lb X, Y \rb
$
is a Fourier-Mukai pair of K3 surfaces of degree 12 as mentioned in \pref{rm:finer}.
\end{remark}

\bibliographystyle{amsalpha}
\bibliography{mainbibs}
\end{document}